\documentclass[a4paper, english,reqno, 11pt]{amsart}
\usepackage{a4wide}
\usepackage{amsmath} 
\usepackage{amsthm} 
\usepackage{amssymb} 

\usepackage{amscd} 

\usepackage{array} 
\usepackage[backref=page,linktocpage]{hyperref} 
\usepackage{caption} %
\usepackage{graphics,graphicx} 
\usepackage{tikz,tikz-cd} 

\usepackage{multirow}
\usepackage{graphicx}

\usepackage{cleveref}
\usepackage[all]{xy}

\usepackage[utf8]{inputenc}
\usepackage[english]{babel}

\usepackage{enumerate} 
\usepackage{xcolor}

\makeatletter
\@namedef{subjclassname@2020}{\textup{2020} Mathematics Subject Classification}
\makeatother

\hypersetup{
    colorlinks = true,
    linkbordercolor = {white},
    linkcolor = {blue},
    anchorcolor = {black},
    citecolor = {blue},
    filecolor = {cyan},
    menucolor = {red},
    runcolor = {cyan},
    urlcolor = {black}
}

\usetikzlibrary{automata}

\theoremstyle{plain}
\newtheorem{theorem}{Theorem}[section]
\newtheorem{corollary}[theorem]{Corollary}
\newtheorem{lemma}[theorem]{Lemma}
\newtheorem{proposition}[theorem]{Proposition}
\newtheorem*{theorem-no-label}{Theorem}

\theoremstyle{definition}

\newtheorem{example}[theorem]{Example}
\newtheorem{remark}[theorem]{Remark}

\DeclareMathOperator{\rk}{rk}

\newcommand{\Chow}{\operatorname{H}}
\newcommand{\Chowaug}{\operatorname{H}^{\operatorname{aug}}}
\DeclareMathOperator{\expsi}{ex\Psi}
\DeclareMathOperator{\expsitilde}{ex\widetilde{\Psi}}

\newcommand{\rank}{\operatorname{rk}}

\newcommand{\zero}{\widehat{0}}
\newcommand{\one}{\widehat{1}}
\newcommand{\aaa}{\mathbf{a}}
\newcommand{\bbb}{\mathbf{b}}

\newcommand{\ZZ}{\mathbb{Z}}

\newcommand{\Dfn}[1]{\emph{\bfseries #1}}
\newcommand{\wt}{{\sf wt}}
\newcommand{\cC}{\mathcal{C}}
\newcommand{\Poin}{\mathcal{\sf Poin}}

\usepackage{todonotes}

\title{Extending the $\aaa\bbb$-index}

\author[E.~Hoster]{Elena Hoster}
\address[E.~Hoster \& C.~Stump]{Ruhr-Universit\"at Bochum, Germany}
\email{\{elena.hoster,christian.stump\}@rub.de}

\author[C.~Stump]{Christian Stump}

\author[L.~Vecchi]{Lorenzo Vecchi}
\address[L.~Vecchi]{Department of Mathematics, KTH Royal Institute of Technology, Sweden}
\email{lvecchi@kth.se}

\date{\today}

\begin{document}

\begin{abstract}
  We prove for finite, graded, bounded posets, that the Poincaré-extended $\aaa\bbb$-index is obtained from the $\aaa\bbb$-index via the $\omega$-transformation.
  This proves a conjecture by Dorpalen-Barry, Maglione, and the second author, and provides a more conceptual approach to $\aaa\bbb$-indices and Chow polynomials beyond $R$-labeled posets.
\end{abstract}

\maketitle

\section{Introduction}

The (Poincaré-)extended $\aaa\bbb$-index of a finite, graded, bounded poset was introduced by Dorpalen-Barry, Maglione, and the second author in~\cite{poincareextended}, enriching the classical $\aaa\bbb$-index with the Poincaré polynomial.
They showed that this polynomial has nonnegative coefficients whenever the poset admits an $R$-labeling.
They then used this to prove a conjecture by Maglione and Voll that the coarse flag Hilbert--Poincaré series has nonnegative numerator polynomial~\cite{MaglioneVoll}.
The second author then showed that the extended $\aaa\bbb$-index also specializes to the Chow polynomial, providing a combinatorial approach to the $\gamma$-positivity of the latter for $R$-labeled posets~\cite{Stump}.
All three authors in different combinations, and also in collaboration with Brändén, used this to give explicit descriptions of the Chow polynomial for uniform matroids and totally nonnegative posets~\cite{2024arXiv241022329H,brandenvecchi} and provide real-rootedness for uniform matroids and for simplicial posets~\cite{2025arXiv250107364B,2025arXiv250815538H}.
Ferroni, Matherne, and the third author then generalized the argument for the $\gamma$-positivity to posets with a nonnegative flag $h$-vector~\cite{ferroni-matherne-vecchi}.

\medskip

This paper provides a way to obtain the extended $\aaa\bbb$-index from the $\aaa\bbb$-index, proving a conjecture from~\cite{poincareextended}.
This ties together multiple of the above results, generalizes the numerical canonical decomposition of the Chow ring from~\cite{ferroni-matherne-vecchi}, and generalizes and simplifies multiple arguments in the literature, see \Cref{thm: exPsi is omega Psi} and \Cref{cor:expsiviabeta,,cor:poincaremain,,cor:ferronietall}.

\section{Definitions and main results}

Let~$P$ be a finite, graded, bounded poset of rank~$n$.
That is,~$P$ is a finite poset with unique minimum element~$\hat 0$ and unique maximum element~$\hat 1$ of rank~$n$ such that $\rank(w)$ is equal to the length of any maximal chain from~$\hat 0$ to~$w$.
Its \Dfn{Möbius function}~$\mu$ is given by $\mu(w, w) = 1$ for all $w\in P$ and $\mu(u,w) = -\sum_{u\leq v < w} \mu(u, v)$ for all $u < w$, and its \Dfn{Poincaré polynomial} is
\[
  \Poin_P(y) = \sum_{w\in P} \mu(\hat{0}, w) \cdot (-y)^{\rank(w)}\,.
\]
A (not necessarily maximal) chain $\cC = \{\cC_1 < \dots < \cC_k < \cC_{k+1}\}$ in~$P$ is an ordered set of pairwise comparable elements.
We assume throughout this paper that all chains end in the maximum element $\cC_{k+1} = \hat 1$.
The \Dfn{chain Poincaré polynomial} then is
\[
  \Poin_{P,\cC}(y)   =\prod_{i = 1}^k \Poin_{[\cC_i,\cC_{i+1}]}(y) \,,
\]
where $[\cC_i,\cC_{i+1}] \subseteq P$ denotes the interval between two consecutive elements in the chain.

\medskip

Let $\ZZ[y]\langle\aaa,\bbb\rangle$ be the polynomial ring in two noncommuting variables $\aaa,\bbb$ with coefficients being polynomials in the variable~$y$.
For a subset $S \subseteq \{0,\dots,n-1\}$, we set $\wt_S(\aaa,\bbb) = w_0\cdots w_{n-1}$ with
\begin{equation*}
  w_k = \begin{cases}
          \bbb        & \text{if } k\in S, \\
          \aaa - \bbb & \text{if } k\notin S\,.
        \end{cases}
\end{equation*}
For a chain $\cC = \{\cC_1 < \dots < \cC_k < \cC_{k+1} = \hat 1\}$, we moreover set $\wt_\cC = \wt_{\{\rank(\cC_1),\dots,\rank(\cC_k)\}}$.
The extended $\aaa\bbb$-index $\expsi_P(y,\aaa,\bbb)$ then is
\[
\expsi_P(y,\aaa,\bbb)
  = \sum_{\cC} \Poin_{P,\cC}(y) \cdot \wt_{\cC}(\aaa,\bbb) \ \in \ZZ[y]\langle\aaa,\bbb\rangle\,,
\]
where the sum ranges over all chains~$\cC = \{\cC_1 <\dots < \cC_{k+1} = \hat 1\}$.
Since~$P$ is bounded, we have $\Poin_P(0) = 1$, implying
\begin{equation}
\label{eq:abindex}
  \expsi_P(0,\aaa,\bbb)
  = \Psi_P(\aaa,\bbb)
  = \sum_{\cC} \wt_\cC(\aaa,\bbb)\,,
\end{equation}
where $\Psi_P(\aaa,\bbb)$ is (a mild variation of) the \Dfn{$\aaa\bbb$-index}  as given, for example, in~\cite{bayer-survey}.
It is sometimes convenient to only consider chains that start in the bottom element~$\hat 0$.
This is achieved by applying the $\iota$-transformation removing the initial letter from every $\aaa\bbb$-monomial,
\[
  \expsitilde_P(y,\aaa,\bbb)    = \iota\big(\expsi_P(y,\aaa,\bbb)\big)\,,\quad
  \widetilde\Psi_P(\aaa,\bbb) = \iota\big(\Psi_P(\aaa,\bbb)\big)\,.
\]

\medskip

In this paper, we prove that the Poincaré-extension of the $\aaa\bbb$-index is obtained from the $\aaa\bbb$-index via the $\omega$-transformation.
For posets admitting an $R$-labeling, this was shown in~\cite[Corollary~2.9]{poincareextended}, where also the general case was conjectured.

\begin{theorem}[{\cite[Conjecture~2.10]{poincareextended}}]
\label{thm: exPsi is omega Psi}
    Let~$P$ be a finite, graded, bounded poset.
    Then
     \begin{align*}
        \expsi_P(y,\aaa,\bbb)      = \omega\big(\Psi_P(\aaa,\bbb)\big)\,, \quad
        \expsitilde_P(y,\aaa,\bbb) = (1+y) \cdot \omega\big(\widetilde\Psi_P(\aaa,\bbb)\big)\,,
    \end{align*}
    where $\omega$ is the transformation that replaces all occurrences of $\aaa\bbb$ with $(1+y)\aaa\bbb + (y+y^2)\bbb\aaa$ and then simultaneously replaces all remaining occurrences of $\aaa$ with $\aaa+y\bbb$ and of $\bbb$ with $\bbb+y\aaa$.
\end{theorem}

As the first immediate corollary, we obtain the following alternative description of the extended $\aaa\bbb$-index in terms of the \Dfn{flag $h$-vector}
\[
    \beta(T) = \sum_{S \subseteq T} (-1)^{|T \setminus S|} \alpha(S)\,,
\]
for $\alpha(S)$ being the \Dfn{flag $f$-vector} counting maximal chains in the subposet of~$P$ with only the ranks in~$S$ selected.
This in particular generalizes the nonnegativity of the extended $\aaa\bbb$-index to posets with nonnegative flag $h$-vector.

\begin{corollary}
\label{cor:expsiviabeta}
    We have
    \[
        \expsi_P (y,\aaa,\bbb) = \sum_{T \subseteq \{0,\dots,n-1\}} \beta(T) \cdot \omega(\operatorname{m}_T)\,,
    \]
    where $\operatorname{m}_T = m_0 \dots m_{n-1}$ with $m_i = \bbb$ if $i \in T$ and $m_i = \aaa$ if $i \notin T$.
    The extended $\aaa\bbb$-index has in particular nonnegative coefficients whenever the flag $h$-vector is nonnegative.
\end{corollary}

This description can be used to bypass the inclusion-exclusion argument for posets admitting an $R$-labeling from~\cite[Section~4]{poincareextended}.
For a nonnegative $R$-labeling~$\lambda$, a maximal chain $\mathcal M = (\mathcal M_0<\dots<\mathcal M_n)$ in~$P$ and a set $E \subseteq \{1,\dots,n\}$, define the sequence $(\lambda_0,\lambda_1,\dots,\lambda_n)$ with $\lambda_0 = 0$ and $\lambda_i = \pm \lambda(\mathcal M_{i-1},\mathcal M_i)$ with the sign being positive if $i \not \in E$ and negative if $i \in E$.
Then set $\operatorname{m}(\mathcal M,E) = m_0\dots m_{n-1}$ to be the $\aaa\bbb$-monomial with $m_i = \bbb$ if $\lambda_{i} > \lambda_{i+1}$ and $m_i = \aaa$ if $\lambda_i \leq \lambda_{i+1}$.

\begin{corollary}[{\cite[Theorem~2.7]{poincareextended}}]
\label{cor:poincaremain}
    Let ~$P$ admit an $R$-labeling.
    Then
    \begin{align*}
        \expsi_P(y,\aaa,\bbb) &= \sum_{\mathcal M, E} y^{\#E} \cdot \operatorname{m}(\mathcal M,E)\,,
    \end{align*}
    where the sum ranges over all \emph{maximal} chains~$\mathcal M$ in~$P$ and all subsets $E \subseteq  \{1, \dots, n\}$.
\end{corollary}

Observe that \Cref{cor:expsiviabeta} together with~\cite[Proposition~5.1]{poincareextended} also yields a version of this corollary for posets not admitting an $R$-labeling,
\[
    \expsi_P(y,\aaa,\bbb) = \sum_{T,E} \beta(T)\cdot y^{\#E} \cdot \operatorname{m}_T(E)\,,
\]
where $\operatorname{m}_T(E)$ is obtained from $\operatorname{m}_T = m_0\dots m_{n-1}$ by replacing $m_i = \aaa$ by $\bbb$ if $i+1 \in E$ and replacing $m_i = \bbb$ by $\aaa$ if $i \in E$, compare~\cite[Equation~(6)]{poincareextended}.
We finally also recover the $\gamma$-positivity of the (augmented) Chow polynomial if~$P$ has a nonnegative flag $h$-vector.
Recall from~\cite[Theorem~2.6]{Stump} that the (augmented) Chow polynomial of~$P$ is obtained from the extended $\aaa\bbb$-index via
\begin{align*}
    \Chowaug_P(x) &= \expsi_P(-x,1,x) \cdot (1-x)^{-n} \,, \\
    \Chow_P(x) &= \expsitilde_P(-x,1,x)\cdot(1-x)^{-n}\,.
\end{align*}
Observe that \cite{Stump} only considered posets admitting an $R$-labeling, but the argument for \cite[Theorem~2.6]{Stump} does not rely on this property and generalizes verbatim.

\begin{corollary}[{\cite[Theorem~4.20]{ferroni-matherne-vecchi}}]
\label{cor:ferronietall}
    The augmented Chow polynomial  has $\gamma$-expansion
    \begin{align*}
      \Chowaug_P(x) &= \sum_{T} \beta(T)\cdot x^{\#T} (1+x)^{n-2\#T}\,,
    \intertext{
    where the sum ranges over all \Dfn{isolated} subsets $T \subseteq \{1,\dots,n-1\}$, i.e., subsets $T$ such that $i \in T$ implies $i+1 \notin T$.
    Analogously, the Chow polynomial 
    has the $\gamma$-expansion
    }
      \Chow_P(x) &= \sum_{T} \beta(T)\cdot x^{\#T} (1+x)^{n-1-2\#T}\,,
    \end{align*}
    where the sum ranges over all isolated subsets $T \subseteq \{2,\dots,n-1\}$.
\end{corollary}

\begin{example}
    Consider the two posets $P$ (on the left) and $Q$ (on the right) depicted below,
    where~$P$ has the given $R$-labeling while~$Q$ does not admit an $R$-labeling.
    \begin{center}
        \begin{tikzpicture}[scale=0.9,
                baseline=(top.base)]
          \node (0) at (0,0) {$\hat{0}$};
        
          \node (1) at (-1.5,1.5) {$u_1$};
          \node (2) at (0,1.5) {$u_2$};
          \node (3) at (1.5,1.5) {$u_3$};
        
          \node (top) at (0,3) {$\hat{1}$};
        
          \draw (0)--(1)--(top) (0)--(2)--(top) (0)--(3)--(top);

          \node[red] at (-0.65,0.9){\tiny $1$};
          \node[red] at ( 0.2,0.9){\tiny $2$};
          \node[red] at ( 1.2,0.9){\tiny $3$};
          \node[red] at (-1.1,2.2){\tiny $2$};
          \node[red] at (-0.2,2.2){\tiny $1$};
          \node[red] at ( 0.55,2.2){\tiny $1$};
    
          \node (0q) at (5,0) {$\hat{0}$};
        
          \node (v1) at (4,1) {$v_1$};
          \node (v2) at (6,1) {$v_2$};
          \node (w1) at (4,2) {$w_1$};
          \node (w2) at (6,2) {$w_2$};
        
          \node (topq) at (5,3) {$\hat{1}$};
        
          \draw (0q)--(v1)--(w1)--(topq) (0q)--(v2)--(w2)--(topq);
        \end{tikzpicture}
    \end{center}
    We start with computing the extended $\aaa\bbb$-index:
    \begin{center}
        \scalebox{0.97}{
        \renewcommand{\arraystretch}{1.2}
      \begin{tabular}[t]{c|c|c}
          $\cC$ in $P$ & $\Poin_{P,\cC}(y)$ & $\wt_\cC$ \\
          \hline
          $\hat 1$ & $1$ &  $(\aaa-\bbb)^2$ \\
          $\hat 0 < \hat 1$ & $1 + 3y + 2y^2$ &  $\bbb(\aaa-\bbb)$ \\
          $u_i < \hat 1$ & $1+y$ & $(\aaa-\bbb)\bbb$ \\
          $\hat 0 < u_i < \hat{1}$ & $(1+y)^2$ & $\bbb^2$
      \end{tabular}
      \qquad
      \begin{tabular}[t]{c|c|c}
          $\cC$ in $Q$ & $\Poin_{Q,\cC}(y)$ & $\wt_\cC$ \\
          \hline
          $\hat 1$ & $1$ &  $(\aaa-\bbb)^3$ \\
          $\hat 0 < \hat 1$ & $1+2y-y^3$ &  $\bbb(\aaa-\bbb)^2$ \\
          $v_i < \hat 1$ & $1+y$ & $(\aaa-\bbb)\bbb(\aaa-\bbb)$ \\
          $\hat 0 < v_i < \hat{1}$ & $(1+y)^2$ & $\bbb^2(\aaa-\bbb)$\\
          $w_i < \hat 1$ & $1+y$ & $(\aaa-\bbb)^2\bbb$ \\
          $\hat 0 < w_i < \hat{1}$ & $(1+y)^2$ & $\bbb(\aaa-\bbb)\bbb$\\
          $v_i < w_i < \hat{1}$ & $(1+y)^2$ & $(\aaa-\bbb)\bbb^2$\\
          $\hat 0 < v_i < w_i < \hat{1}$ & $(1+y)^3$ & $\bbb^3$
      \end{tabular}
      }
    \end{center}
    We obtain
      \begin{align*}
        \expsi_P(y,\aaa,\bbb)
          &= (\aaa-\bbb)^2\! +\! (1 + 3y + 2y^2)\bbb(\aaa-\bbb)\! +\! 3\cdot (1+y)(\aaa-\bbb)\bbb\! +\! 3\cdot(1+y)^2\bbb^2 \\
          &= \aaa^2 + (3y+2y^2)\bbb\aaa + (2+3y)\aaa\bbb + y^2\bbb^2 \,, \\
        \expsitilde_P(y,\aaa,\bbb)
          &= (1+3y+2y^2)\aaa + (2+3y+y^2)\bbb \,, \\
        \Psi_P(\aaa,\bbb) &= \aaa^2 + 2\aaa\bbb \,, \\
        \widetilde\Psi_P(\aaa,\bbb) &= \aaa + 2\bbb \,, \\
        \Chowaug_P(x) &= 1 + 4x + x^2 \,, \\
        \Chow_P(x)    &= 1+x \,,
    \intertext{and}
        \expsi_Q(y,\aaa,\bbb)
          &= (\aaa-\bbb)^3\! +\! (1 + 2y - y^3)\bbb(\aaa-\bbb)^2\! +\! 2\cdot (1+y)\big( (\aaa-\bbb)\bbb(\aaa-\bbb)+(\aaa-\bbb)^2\bbb\big)\! \\
          &\qquad +\! 2\cdot(1+y)^2\big(\bbb^2(\aaa-\bbb)+\bbb(\aaa-\bbb)\bbb+(\aaa-\bbb)\bbb^2\big)\! +\! 2\cdot(1+y)^3\bbb^3  \\
          &= \aaa^{3} + (1 + 2 y) \aaa^{2}\bbb + (1 + 2 y) \aaa \bbb \aaa + (-1 + 2 y^{2}) \aaa \bbb^{2} + (2y-y^{3}) \bbb \aaa^{2} \\
          &\qquad+ (y^{2} + 2 y^{3}) \bbb \aaa \bbb + (y^{2} + 2 y^{3}) \bbb^{2}\aaa + y^{3} \bbb^{3} \,, \\
        \expsitilde_Q(y,\aaa,\bbb)
          &= (1 + 2y -y^{3}) \aaa^{2} + (1 + 2y+2y^2 + y^{3}) \aaa \bbb \\
          &\qquad+ (1 + 2 y + 2 y^2 + y^{3}) \bbb \aaa + (-1 + 2y^2 + y^{3}) \bbb^{2} \,, \\
        \Psi_Q(\aaa,\bbb) &= \aaa^{3} + \aaa^{2}\bbb + \aaa \bbb \aaa - \aaa \bbb^{2} \,, \\
        \widetilde\Psi_Q(\aaa,\bbb) &= \aaa^{2} + \aaa\bbb + \bbb \aaa - \bbb^{2} \,, \\
        \Chowaug_Q(x) &= 1 + 5x + 5x^2 + x^3 \,, \\
        \Chow_Q(x)    &= 1 + 3x + x^2 \,.
    \end{align*}

    One may also compute the extended $\aaa\bbb$-index using the flag $h$-vector.
    For the poset~$P$, we obtain
    \begin{center}
        \scalebox{0.97}{
        \renewcommand{\arraystretch}{1.2}
      \begin{tabular}[t]{c|c|c|c|c}
          $T\subseteq \{0,1\}$ & $\alpha(T)$ & $\beta(T)$ & $\operatorname{m}_T$ & $\omega (\operatorname{m}_T )$\\
          \hline
          $\emptyset$ & $1$ & $1$ & $\aaa^2$ & $(\aaa+y\bbb)^2$ \\
          $\{1\}$ & $3$ & $2$ & $\aaa\bbb$   & $(1+y)\aaa\bbb + (y+y^2)\bbb\aaa$ \\
      \end{tabular}
      }
    \end{center}

    \noindent where we ignored the sets $T \ni 0$ as $\beta_P(T) = 0$ in this case.
    For the poset $Q$, we obtain
    \begin{center}
        \scalebox{0.97}{
        \renewcommand{\arraystretch}{1.2}
      \begin{tabular}[t]{c|c|c|c|c}
          $T\subseteq \{0,1,2\}$ & $\alpha(T)$ & $\beta(T)$ & $\operatorname{m}_T$ & $\omega (\operatorname{m}_T )$\\
          \hline
          $\emptyset$ & $1$ & $1$ & $\aaa^3$ & $(\aaa+y\bbb)^3$ \\
          $\{1\}$ & $2$ & $1$ & $\aaa\bbb\aaa$   & $(1+y)\aaa\bbb\aaa + (y+y^2)(\bbb\aaa^2 + \aaa\bbb^2)+(y^2+y^3)\bbb\aaa\bbb$ \\
          $\{2\}$ & $2$ & $1$ & $\aaa^2\bbb$   & $(1+y)\aaa^2\bbb + (y+y^2)(\aaa\bbb\aaa + \bbb\aaa\bbb)+(y^2+y^3)\bbb^2\aaa$ \\
          $\{1,2\}$ & $2$ & $-1$ & $\aaa\bbb^2$  & $(1 + y) \aaa \bbb^{2} + (y + y^2) (\aaa \bbb \aaa + \bbb \aaa \bbb )+ (y^{2} + y^{3}) \bbb \aaa^{2}  $\\
      \end{tabular}
      }
    \end{center}
    Applying \Cref{cor:expsiviabeta} returns the extended $\aaa\bbb$-index as computed before.
    For the $R$-labeling~$\lambda$ of~$P$ given in {\color{red} red}, the $\aaa\bbb$-monomials $\operatorname{m}(\mathcal{M},E)$ used in \Cref{cor:poincaremain} are
    \begin{center}
    \renewcommand{\arraystretch}{1.2}
    \begin{tabular}{c|cccc|cccc}
    \multirow{2}{*}{$\mathcal{M}$} & \multicolumn{4}{c}{$\lambda(\mathcal{M}, E)$} & \multicolumn{4}{c}{$\operatorname{m}(\mathcal{M}, E)$} \\
        & $\emptyset$   &  $\{1\}$  &  $\{2\}$  & $\{1,2\}$ 
        & $\emptyset$   &  $\{1\}$  &  $\{2\}$  & $\{1,2\}$  \\ \hline
    $\hat 0 < u_1 < \hat 1$ 
        & $(0,1,2)$ &  $(0,-1,2)$  &  $(0,1,-2)$  & $(0,-1,-2)$ 
        & $\aaa\aaa$ &  $\bbb\aaa$  &  $\aaa\bbb$  & $\bbb\bbb$ \\ \hline
    $\hat 0 < u_2 < \hat 1$ 
        & $(0,2,1)$ & $(0,-2,1)$  & $(0,2,-1)$ & $(0,-2,-1)$ 
        & $\aaa\bbb$ & $\bbb\aaa$  & $\aaa\bbb$ & $\bbb\aaa$ \\ \hline
    $\hat 0 < u_3 < \hat 1$ 
        & $(0,3,1)$ & $(0,-3,1)$  & $(0,3,-1)$ & $(0,-3,-1)$ 
        & $\aaa\bbb$ & $\bbb\aaa$  & $\aaa\bbb$ & $\bbb\aaa$ 
    \end{tabular}
    \end{center}
    We leave the remaining computations of the extended $\aaa\bbb$-indices and of the (augmented) Chow polynomials to the reader.
\end{example}

\section{Proofs}

Throughout this section, we assume the poset to not be trivial, i.e., having rank $n \geq 1$.
We start with the following recursive structure of the extended $\aaa\bbb$-index.
\begin{proposition}
\label{lemma:recursion expsi}
    We have
    \begin{align*}
    \expsi_P(y,\aaa,\bbb) &= (\aaa-\bbb)^{n} + \Poin_P(y)\cdot \bbb (\aaa-\bbb)^{n-1} \\[5pt]
             &\qquad + \sum_{\zero < w < \one} \big( (\aaa-\bbb)^{\rk(w)} + \Poin_{[\zero,w]}(y) \cdot \bbb (\aaa-\bbb)^{\rk(w) -1} \big) \bbb \cdot \expsitilde_{[w,\one]}(y,\aaa,\bbb) \,, \\[15pt]
    \expsitilde_P(y,\aaa,\bbb) &= \Poin_P(y)\cdot (\aaa-\bbb)^{n-1} \\[5pt]
                  &\qquad + \sum_{\zero < w < \one}\Poin_{[\zero,w]}(y)\cdot(\aaa-\bbb)^{\rk(w) -1}\bbb \cdot \expsitilde_{[w,\one]}(y,\aaa,\bbb) \,.
    \end{align*}
\end{proposition}
\begin{proof}
    The first two summands in the formula for $\expsi_P(y,\aaa,\bbb)$ come from the chain $\cC = \{\hat 1\}$ and from the chain $\cC = \{\hat 0 < \hat 1\}$.
    The third sum then sums over all elements $\hat 0 < w < \hat 1$, and combines the chains starting in~$w$ and starting in~$\hat 0 < w$.
    The formula for $\expsitilde_P(y,\aaa,\bbb)$ is analogous, where we observe that only the chains starting in~$\hat 0$ are present, and with the first letter removed.
\end{proof}

Now evaluating $y\mapsto 0$ yields the analogous recursive structure for the $\aaa\bbb$-index.

\begin{corollary}
\label{lemma: recursive Psi}
    We have
    \begin{align*}
        \Psi_P(\aaa,\bbb) 
        &= \aaa (\aaa - \bbb)^{n-1} + \sum_{\zero < w < \one} \aaa (\aaa-\bbb)^{\rk(w) -1}\bbb \cdot \widetilde{\Psi}_{[w,\one]}(\aaa,\bbb) \,, \\[15pt]
        \widetilde\Psi_P(\aaa,\bbb) 
        &= (\aaa - \bbb)^{n-1} + \sum_{\zero < w < \one} (\aaa-\bbb)^{\rk(w) -1}\bbb \cdot \widetilde{\Psi}_{[w,\one]}(\aaa,\bbb) \,.
    \end{align*}
\end{corollary}

We prove \Cref{thm: exPsi is omega Psi} recursively by applying $\omega$ to the recursive structure in \Cref{lemma: recursive Psi} and show that it agrees with the recursive structure in \Cref{lemma:recursion expsi}.

\begin{proposition}
\label{prop:omegapsi}
    We have
    \begin{align}
        \omega\big(\Psi_P(\aaa,\bbb)\big) &= (\aaa-(-y)^n\bbb)(\aaa-\bbb)^{n-1} 
    \label{eq:omegapsi}\\
    &\ + \sum_{\zero < w < \one}
    \big( \aaa(\aaa-\bbb)^{\rk(w) -1} \bbb - (-y)^{\rk(w)}\bbb (\aaa-\bbb)^{\rk(w) -1} \aaa \big) (1+y) \omega\big(\widetilde\Psi_{[w,\one]}(\aaa,\bbb)\big) \notag
    \intertext{and}
        \omega\big(\widetilde\Psi_P(\aaa,\bbb)\big) &= 
        \tfrac{1-(-y)^n}{1+y} \cdot (\aaa - \bbb)^{n-1} \label{eq:omegapsitilde}\\
        &\quad + \sum_{\zero < w < \one}
    (\aaa - \bbb)^{\rk(w) -1} (\bbb- (-y)^{\rk(w)}\aaa) \cdot \omega\big(\widetilde\Psi_{[w,\one]}(\aaa,\bbb)\big) \,. \notag
    \end{align}
\end{proposition}

This proposition follows directly from the evaluations of the $\omega$-transformation in \Cref{lemma: omega evaluations} below.
Before proceeding, we relate both evaluations to the \emph{numerical canonical decomposition} from~\cite{ferroni-matherne-vecchi}.

\begin{remark}
\label{rem:omegaeval}
    For $f \in \ZZ\langle\aaa,\bbb\rangle$, write $\omega_{\operatorname{ev}}(f) = \omega(f)\big|_{y \mapsto -x,\ \aaa\mapsto 1,\ \bbb\mapsto x} \in \ZZ[x]$ to be the evaluation from~\cite[Theorem~2.6]{Stump}.
    Together with \Cref{thm: exPsi is omega Psi}, we may write
    \[
    (1-x)^n\cdot\Chow_P(x) = \expsitilde_P(-x,1,x) = (1-x)\cdot\omega_{\operatorname{ev}}\big(\widetilde\Psi_P(\aaa,\bbb)\big)\,,
    \]
    where we used the identity 
    \begin{align}
        \iota \circ \omega = (1+y) \cdot \omega\circ \iota \label{eq:omegaiota}
    \end{align}
    which is immediate from the definition.
    Applying the same specialization to  \Cref{prop:omegapsi}\eqref{eq:omegapsitilde} lets us recover the \emph{numerical canonical decomposition} from \cite[Theorem~3.9(15)]{ferroni-matherne-vecchi},
    \begin{align*}
        &(1-x)^{n-1}\Chow_P(x) =\omega_{\operatorname{ev}}\big(\widetilde\Psi_P(\aaa,\bbb)\big)\\[5pt]
        &\qquad= \frac{1-x^n}{1-x}(1-x)^{n-1} + \sum_{\zero < w < \one} (1-x)^{\rk(w) -1}(x-x^{\rk(w)})\cdot (1-x)^{n-\rk(w) - 1}\Chow_{[w,\one]}(x) \\
        &\qquad= (1-x)^{n-1}\left[ \frac{1-x^{n}}{1-x} + \sum_{\zero< w < \one}\frac{x-x^{\rk(w)}}{1-x}\Chow_{[w,\one]}(x) \right].
    \end{align*}
    Similarly, \Cref{prop:omegapsi}\eqref{eq:omegapsi} lets us write
    \[
    \Chowaug_P(x) = \frac{1-x^{n+1}}{1-x} + \sum_{\zero < w < \one}\frac{x-x^{\rk(w)+1}}{1-x}\Chow_{[w,\one]}(x).
    \]
\end{remark}

\begin{lemma}
\label{lemma: omega evaluations}
    We have
    \begin{align*}
        \omega((\aaa-\bbb)^k) &= \tfrac{1-(-y)^{k+1}}{1+y} \ (\aaa-\bbb)^k \,, \\[5pt]
        \omega((\aaa-\bbb)^k \bbb) &= (\aaa-\bbb)^k(\bbb-(-y)^{k+1}\aaa) \,, \\[5pt]
        \omega(\aaa(\aaa-\bbb)^k) &= (\aaa-(-y)^{k+1} \bbb) (\aaa -\bbb)^k \,, \\[5pt]
        \omega(\aaa(\aaa-\bbb)^k \bbb) &= (1+y) \big(\aaa(\aaa -\bbb)^k\bbb -(-y)^{k+1} \bbb (\aaa -\bbb)^k \aaa \big) \,.
    \end{align*}
\end{lemma}
\begin{proof}
    We start with the first two identities by induction on~$k$.
    For $k=1$, we have
    \begin{align*}
        \omega\big(\aaa -\bbb\big) &= \omega(\aaa) - \omega(\bbb)= (1-y)(\aaa-\bbb) \\
        \omega\big((\aaa -\bbb)\bbb \big) \,,
        &= \omega(\aaa\bbb) -\omega(\bbb\bbb)
        = (\aaa - \bbb) (\bbb - y^2 \aaa) \,.
    \end{align*}
    Assuming the first two identities for~$k \geq 1$, we obtain
    \begin{align*}
        \omega\big((\aaa - \bbb)^{k+1}\big) 
        &= \omega\big((\aaa - \bbb)^k \aaa\big) - \omega\big((\aaa - \bbb)^k \bbb\big) \\
        &= \omega\big((\aaa - \bbb)^k\big) \cdot \omega(\aaa) - \omega\big((\aaa - \bbb)^k\bbb\big) \\
        &= \tfrac{1 - (-y)^k}{1 + y}(\aaa - \bbb)^k (\aaa + y \bbb) - (\aaa - \bbb)^k ( \bbb - (-y)^{k+1} \aaa) \\
        &= \tfrac{1 - (-y)^{k+2}}{1 + y}(\aaa - \bbb)^{k+1}  \\
        \intertext{and}
        \omega\big((\aaa - \bbb)^{k+1} \bbb\big)
        &= \omega((\aaa - \bbb)^k \aaa \bbb) - \omega((\aaa - \bbb)^k \bbb \bbb) \\
        &= \omega((\aaa - \bbb)^k) \cdot \omega(\aaa \bbb) - \omega((\aaa - \bbb)^k\bbb) \cdot \omega(\bbb) \\
        &= (\aaa - \bbb)^k (1 - (-y)^{k+1}) \big( \aaa \bbb + y\bbb \aaa \big) - (\aaa - \bbb)^k (\bbb - (-y)^{k+1} \aaa)(\bbb + y \aaa) \\
        &= (\aaa - \bbb)^{k+1} \big( \bbb - (-y)^{k+2} \aaa \big) \,.
    \intertext{
    Here, we used the induction hypotheses in the third line of each equation.
    We finally deduce the third and forth identities using the first and second identity:
    }
        \omega\big(\aaa (\aaa-\bbb)^k\big)
        &= \omega(\bbb) \omega((\aaa-\bbb)^k) +\omega((\aaa-\bbb)^{k+1})\\
        &= \tfrac{1}{1+y} \bigg( \bbb(\aaa-\bbb)^k(1-(-y)^{k+1}) +\aaa(\aaa-\bbb)^k(y+(-y)^{k+2})\\
           &\qquad \qquad + \aaa(\aaa-\bbb)^k (1-(-y)^{k+2}) - \bbb(\aaa-\bbb)^k(1-(-y)^{k+2}) \bigg) \\
        &= (\aaa-(-y)^{k+1} \bbb) (\aaa-\bbb)^k 
    \intertext{and}
        \omega\big(\aaa (\aaa-\bbb)^k \bbb\big)
        &= \omega(\bbb) \cdot \omega\big((\aaa-\bbb)^k\bbb\big) +\omega\big((\aaa-\bbb)^{k+1}\bbb\big)\\
        &= (1+y) \aaa(\aaa - \bbb)^k\bbb - (1+y) (-y)^{k+1} \bbb(\aaa-\bbb)^k \aaa \,. \qedhere
    \end{align*}
\end{proof}

\begin{proof}[Proof of \Cref{prop:omegapsi}]
    \Cref{prop:omegapsi}\eqref{eq:omegapsitilde} follows from the first two identities in \Cref{lemma: omega evaluations}, and~\eqref{eq:omegapsi} from the latter two.
\end{proof}

The key ingredient for proving \Cref{thm: exPsi is omega Psi} is the following identity, which is a standard identity in the incidence algebra.

\begin{lemma}
\label{lemma:sum poincare}
    \[\sum_{w \in P} (-y)^{\rk(w)} \Poin_{[w,\one]}(y) = 1\, .\]
    In particular,
    \[
    \Poin_P(y) = 1 - (-y)^n - \sum_{\zero < w < \one}(-y)^{\rk(w)}\Poin_{[w,\one]}(y)\, . 
    \]
\end{lemma}
\begin{proof}
    It is known that $\sum_w \chi_{[w,\one]}(y) = y^n$, this can be proven for example by applying the M\"obius inversion formula.
    The statement then follows with $\Poin_P(y) = (-y)^n \chi_P(-\tfrac{1}{y})$.
\end{proof}
\begin{proof}[Proof of Theorem \ref{thm: exPsi is omega Psi}]
    We first use \Cref{lemma:sum poincare} in the recursive definition of $\expsi_P(y,\aaa,\bbb)$ from \Cref{lemma:recursion expsi}.
    \begin{align*}
    \expsi_P (y,\aaa,\bbb)
    &= \ (\aaa - \bbb)^n + \Poin_P(y) \cdot \bbb (\aaa - \bbb)^{n-1} \\
        &\qquad+ \sum_{\zero < w < \one} 
        (\aaa - \bbb)^{\rk(w)} \cdot \bbb \cdot \expsitilde_{[w,\one]}(y,\aaa,\bbb) \\
        &\qquad + \sum_{\zero < w < \one} 
        \Poin_{[\zero,w]}(y) \cdot \bbb (\aaa - \bbb)^{\rk(w)-1} \cdot \bbb \cdot \expsitilde_{[w,\one]}(y,\aaa,\bbb) \\
        &= (\aaa - \bbb)^n + (1 - (-y)^n) \cdot \bbb(\aaa - \bbb)^{n-1} \\
        &\qquad- \sum_{0 < w <\one} (-y)^{\rk(w)} \Poin_{[w,\one]}(y) \cdot \bbb (\aaa - \bbb)^{n-1} \\
    &\qquad + \sum_{0 < w <\one} (\aaa - \bbb)^{\rk(w)}\cdot \bbb \cdot \expsitilde_{[w,\one]}(y,\aaa,\bbb) \\
    &\qquad+ \sum_{0 < w <\one} \bbb (\aaa - \bbb)^{\rk(w)-1}\cdot \bbb \cdot \expsitilde_{[w,\one]}(y,\aaa,\bbb) \\
    &\qquad - \sum_{0 < w <\one} (-y)^{\rk(w)} \bbb (\aaa - \bbb)^{\rk(w)-1}\cdot \bbb \cdot \expsitilde_{[w,\one]}(y,\aaa,\bbb) \\
    &\qquad - \sum_{0 < u < w <\one} (-y)^{\rk(u)} \Poin_{[u,w]}(y)\cdot \bbb (\aaa - \bbb)^{\rk(w)-1}\cdot \bbb \cdot \expsitilde_{[w,\one]}(y,\aaa,\bbb) \,.
    \end{align*}
    The final sum may be rewritten as
    \begin{align*}
    &\sum_{0 < u < w <\one}(-y)^{\rk(u)} \Poin_{[u,w]}(y) \cdot\bbb (\aaa - \bbb)^{\rk(w)-1}\cdot \bbb \cdot \expsitilde_{[w,\one]}(y,\aaa,\bbb)\\
        &\quad= \ \sum_{0 < u < \one} (-y)^{\rk(u)} \cdot \bbb (\aaa - \bbb)^{\rk(u)} 
            \!\! \sum_{u < w < 1} \Poin_{[u,w]}(y) \cdot (\aaa - \bbb)^{\rk(w) - \rk(u) -1} \bbb \cdot \expsitilde_{[w,\one]}(y,\aaa,\bbb)\\
        &\quad= \ \sum_{0 < u < \one} (-y)^{\rk(u)}\cdot \bbb (\aaa - \bbb)^{\rk(u)} 
            \big( \expsitilde_{[u,1]}(y,\aaa,\bbb) - \Poin_{[u,1]}(y) \cdot (\aaa-\bbb)^{n-1-\rk(u)} \big) \,.
    \end{align*}
    A few more term orderings and cancellations give
    \begin{align*}
    \expsi_P(y,\aaa,\bbb) 
    &= \ \aaa(\aaa - \bbb)^{n-1} - (-y)^n \cdot \bbb(\aaa - \bbb)^{n-1} \\
    &\qquad - \sum_{0 <w < \one} (-y)^{\rk(w)} \Poin_{[w,\one]}(y) \cdot \bbb(\aaa - \bbb)^{n-1} \\
    &\qquad + \sum_{0 < w <\one} \aaa(\aaa - \bbb)^{\rk(w)-1}\cdot \bbb \cdot \expsitilde_{[w,\one]}(y,\aaa,\bbb) \\
    &\qquad - \sum_{0 < w <\one} (-y)^{\rk(w)} \cdot \bbb (\aaa - \bbb)^{\rk(w)-1}\cdot \bbb \cdot \expsitilde_{[w,\one]}(y,\aaa,\bbb)\\
    &\qquad - \sum_{0 < u < \one} (-y)^{\rk(u)} \cdot \bbb (\aaa - \bbb)^{\rk(u) -1}\cdot (\aaa - \bbb) \cdot \expsitilde_{[u,1]}(y,\aaa,\bbb) \\
    &\qquad + \sum_{0 < u < \one} (-y)^{\rk(u)} \Poin_{[u,\one]}(y)\cdot \bbb (\aaa - \bbb)^{n-1}\\
    &= (\aaa- (-y)^n \bbb) (\aaa - \bbb)^{n-1} \\
    &\qquad + \sum_{0 < w <\one} \aaa(\aaa - \bbb)^{\rk(w)-1}\cdot \bbb \cdot \expsitilde_{[w,\one]}(y,\aaa,\bbb) \\
    &\qquad- \sum_{0 < u < \one} (-y)^{\rk(u)} \cdot \bbb (\aaa - \bbb)^{\rk(u) -1}\cdot \aaa \cdot \expsitilde_{[u,1]}(y,\aaa,\bbb) \\
    &= (\aaa-(-y)^n\bbb)(\aaa-\bbb)^{n-1} \\
    &\qquad +  \sum_{\zero < w < \one} 
        \big( \aaa(\aaa-\bbb)^{\rk(w) -1} \bbb - (-y)^{\rk(w)}  \bbb (\aaa-\bbb)^{\rk(w) -1} \aaa \big)  \cdot 
        \expsitilde_{[w,\one]}(y,\aaa,\bbb) \,.
\end{align*}
By induction on the rank of the poset, we may assume
\[
  \expsitilde_{[w,\one]}(y,\aaa,\bbb) = (1+y) \cdot \omega\big( \Psi_{[w,\one]}(\aaa,\bbb)\big)\,,
\]
which is clearly satisfied if the poset~$P$ has rank~$1$.
We have therefore obtained the right hand side of \eqref{eq:omegapsi}, which proves that 
\[
 \expsi_P(y, \aaa,\bbb) = \omega\big( \Psi_P(\aaa,\bbb )\big) \,.
\]
Applying $\iota$ on both sides finally gives
\[
    \expsitilde_P(y, \aaa,\bbb) = \iota \circ \omega\big( \Psi_P(\aaa,\bbb )\big)
     = (1+y) \cdot \omega\big( \widetilde \Psi_P(\aaa,\bbb )\big) \,,
\]
where we used $\iota \circ \omega = (1+y) \cdot \omega\circ \iota$ from~\eqref{eq:omegaiota}.
\end{proof}

We conclude with the proofs of the corollaries.

\begin{proof}[Proof of \Cref{cor:expsiviabeta}]
    The $\aaa\bbb$-index has the expansion
    \[
        \Psi_P(\aaa,\bbb) = \sum_{T \subseteq \{0,\dots,n-1\}} \beta(T)\cdot \operatorname{m}(T)\,,
    \]
    see e.g.~\cite[Remark~4.4]{poincareextended}.
    The statement follows from the linearity of the $\omega$-trans\-for\-mation.
\end{proof}

\begin{proof}[Proof of \Cref{cor:poincaremain}]
    This follows immediately from \Cref{thm: exPsi is omega Psi} using~\cite[Proposition~5.1]{poincareextended}.
\end{proof}

\begin{proof}[Proof of \Cref{cor:ferronietall}]
    Applying $\omega_{\operatorname{ev}}$, as defined in \Cref{rem:omegaeval}, to the expression in \Cref{cor:expsiviabeta} yields the conclusion.
    Here, we observe that the summand corresponding to $T$ evaluates to~$0$ if $T$ is not isolated and to $x^{\#T}(1+x)^{n-2\#T}$ otherwise.
    This completes the proof for the augmented Chow polynomial, and we refer to the proof of \cite[Theorems~1.1 \& 1.2]{Stump} for more details on this evaluation.
    The statement for the Chow polynomial is completely analogous.
\end{proof}

\includegraphics[width=0.1px, keepaspectratio]{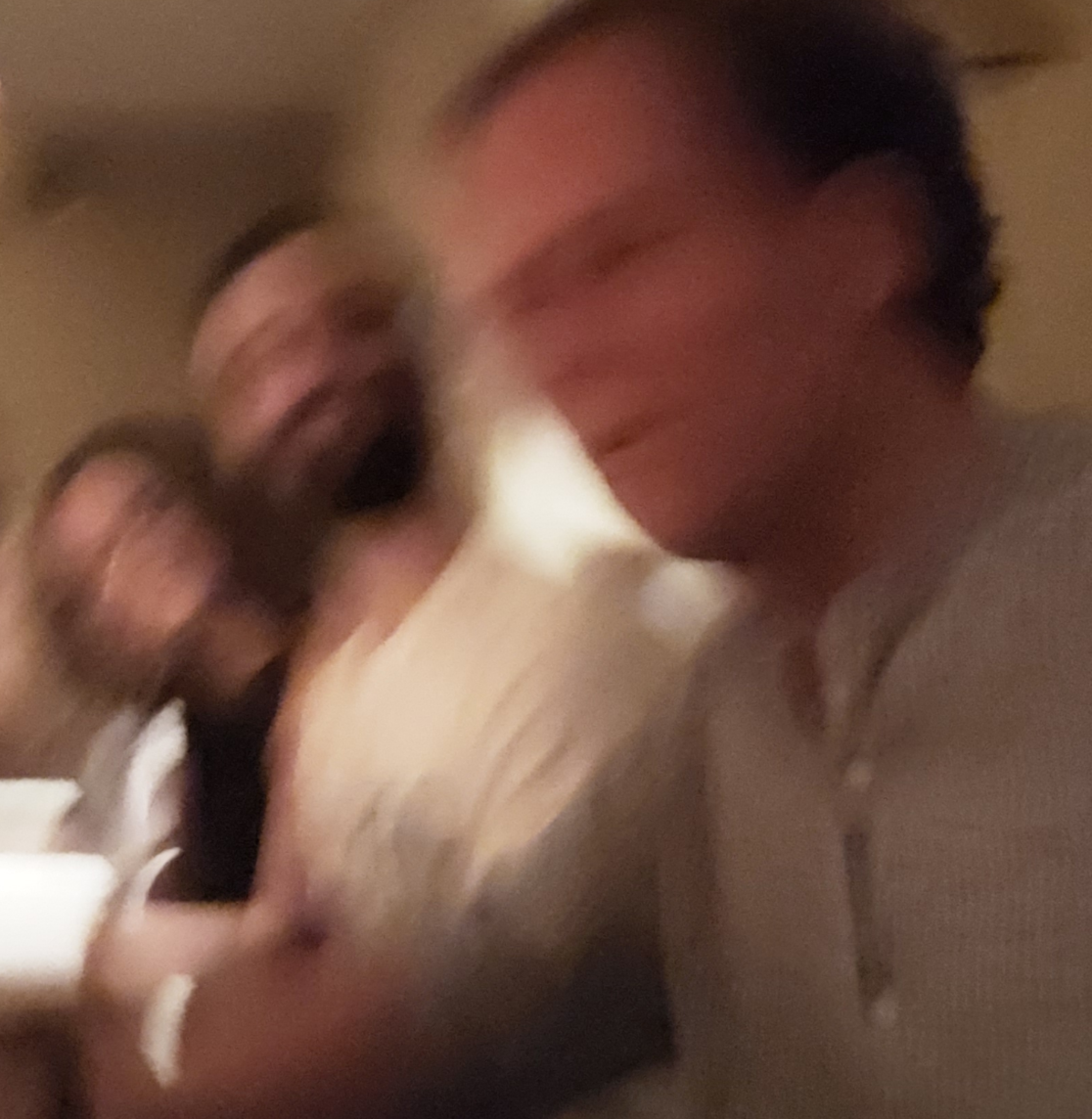}

\section*{Acknowledgements}

The authors would like to thank Galen Dorpalen-Barry and Josh Maglione for many interesting and enlightening discussions about extending the $\aaa\bbb$-index.

\bibliographystyle{alpha}
\bibliography{bibliography}

\end{document}